\documentclass[11pt]{article}
\usepackage{amsfonts,latexsym,amssymb,amsthm,amsmath,cases,mathrsfs}
\usepackage{paralist}
\usepackage{empheq}

\allowdisplaybreaks
\newtheorem{theorem}{Theorem}[section]

\newtheorem{thm}{Theorem}

\newtheorem{lemma}[theorem]{Lemma}

\theoremstyle{definition}

\newtheorem{rmk}{Remark}

\newcommand{\Bell}{\mathcal{B}_\ell}
\newcommand{\All}{\mathcal{A}_\ell}

\newcommand{\be}{\begin{equation}}
\newcommand{\ee}{\end{equation}}
\newcommand{\bsubeq}{\begin{subequations}}
	\newcommand{\esubeq}{\end{subequations}}
\renewcommand{\div}{\text{div}}
\newcommand{\ds}{\displaystyle}

\newcommand{\calL}{{\mathcal{L}}}

\newcommand{\calR}{{\mathcal{R}}}

\newcommand{\calB}{{\mathcal{B}}}
\newcommand{\calD}{{\mathcal{D}}}

\newcommand{\calA}{{\mathcal{A}}}

\newcommand{\calC}{{\mathcal{C}}}

\newcommand{\calM}{{\mathcal{M}}}

\newcommand{\calS}{{\mathcal{S}}}

\newcommand{\calV}{{\mathcal{V}}}

\newcommand{\BR}{\mathbb{R}}

\newcommand{\BN}{\mathbb{N}}

\newcommand{\wti}{\widetilde}

\newcommand{\bpm}{\begin{pmatrix}}
	\newcommand{\epm}{\end{pmatrix}}

\newcommand{\bbm}{\begin{bmatrix}}
	\newcommand{\ebm}{\end{bmatrix}}

\numberwithin{equation}{section}
\numberwithin{thm}{section}
\numberwithin{rmk}{section}
\numberwithin{prop}{section}

\newcommand{\bs}[1]{\boldsymbol{#1}}

\newcommand\rfrac[2]{{}^{#1}\!/_{#2}}

\newcommand{\norm}[1]{\left\lVert#1\right\rVert}
\newcommand{\abs}[1]{\left\lvert#1\right\rvert}

\newcommand{\ip}[2]{\langle #1, #2 \rangle}
\newcommand{\ipp}[2]{( #1, #2 )}

\newcommand{\nin}{\noindent}

\newcommand{\calP}{\mathcal{P}}

\newcommand{\tcb}[1]{\textcolor{black}{#1}}
\newcommand{\tcr}[1]{\textcolor{black}{#1}}

\usepackage{todonotes}

\begin{document}
	\title{The Existence and Dimension of the Attractor for a 3D Flow of a Non-Newtonian Fluid subject to Dynamic Boundary Conditions \thanks{The authors were supported by the project No. 20-11027X financed by GA\v{C}R.}}
	
	
	\author{Dalibor Pra\v{z}\'{a}k \thanks{Charles University, Faculty of Mathematics and Physics, Sokolovsk\'{a} 83, Praha, CZ 186 75, Czech Republic. (prazak@karlin.mff.cuni.cz).}
		\and Buddhika Priyasad \thanks{Charles University, Faculty of Mathematics and Physics, Sokolovsk\'{a} 83, Praha, CZ 186 75, Czech Republic. (priyasad@karlin.mff.cuni.cz).}}

	\date{}
	\maketitle
	
	\begin{abstract}
		\nin We consider non-Newtonian incompressible 3D fluid of
		Ladyzhenskaya type, in the setting of the dynamic boundary
		condition. Assuming sufficient growth rate of the stress tensor
		with respect to the velocity gradient, we establish explicit
		dimension \tcb{estimate} of the global attractor in terms of the physical \tcb{parameters} of the problem.
	\end{abstract}
	
	\section{Introduction.}
	
	The existence of global attractor, its finite-dimensionality, and
	possibly even the construction of a finite-dimensional exponential
	attractor belong to prototypical results of the dynamical
	theory of nonlinear evolutionary PDEs. 
	These goals are often attained, as long as the system
	is well-posed and dissipative. The literature being too extensive 
	to quote, let us
	mention the basic monographs 
	\cite{ChVi02},  \cite{CLW:2014}, \cite{IC:1980}, \cite{Ro2011}, \cite{Te97}. On the other hand, an explicit dimension estimate of the
	attractor is a different matter, requiring additional tools from
	functional analysis, and considerably more demanding in view of the
	regularity of the underlying solution semigroup.
	
	\par
	Focusing to the incompressible Navier-Stokes equations as a model
	problem, one can say that in 2D, the problem of the attractor dimension
	is rather well understood. Reasonable upper estimates are available for
	various domains, even unbounded ones, and the results are known to be sharp
	for the torus\tcr{, see} recent paper \cite{IPZ:2016} and the references therein. 
	For the 3D case, weak solutions \tcr{exist} globally, but the uniqueness remains a
	famous open problem even for the torus. One can still define (sort of) an attractor,
	but nothing can be said about its dimension. Consequently, various regularizations of
	the problem, more or less well-motivated physically, have been proposed, for
	which these problems were then successfully addressed, cf. for
	example \cite{IKZ:2022} for the so-called Euler-Bardina
	regularization.
	
	\par
	In the present paper, we consider one such classical modification,
	going back to Ladyzhenskaya \cite{La67}, where additional gradient integrability
	is induced  by a non-linear modification of the viscous stress tensor via the
	$r$-Laplacian type term $\ds \abs{\bs{Du}}^{r-2}\bs{Du}$. Thus, one the one hand, the
	problem becomes well-posed \tcb{in 3D} for values only slightly above the
	NSE-critical value $r=2$. On the other hand, such \tcr{a} highest order
	nonlinearity brings additional complications to the analysis, as in
	particular higher regularity of weak solutions is difficult to obtain in
	dimensions other than two. Note that this so-called Ladyzhenskaya model
	is well-motivated physically \cite{MaRa05}.
	
	\par
	The problem of the attractor dimension, and more generally, the
	structural complexity of the dynamics, is presumable highly sensitive to
	the adopted boundary condition. Motivated by this, we further generalize
	our setting to allow for a non-linear evolution on $\partial\Omega$,
	which is driven by the normal stress force of the fluid, exerted across
	the boundary. Our result is new in particular by providing an explicit 
	(asymptotic) dimension estimate for 3D fully non-linear problem, while remaining
	in the setting of weak solutions only.
	
	\par
	Let us finally mention some related publications and results concerning
	our model, i.e. the Ladyzhenskaya $r$-fluid. For basic existence and
	uniqueness theory of weak solutions under dynamic boundary conditions,
	see recent paper  \cite{ABM:2021}, cf. also \cite{Mar:2019}.
	Existence of finite-dimensional exponential attractors was recently 
	established in a rather general setting, but without explicit dimension estimates
	\cite{PZ:2022b}. Concerning the Dirichlet boundary conditions,
	explicit dimension estimates in 3D setting  were previously
	obtained in \cite{BEKP:2009}, to which the current paper is
	a direct generalization. Improved dimension estimates, based
	on the volume contraction method, were also obtained 
	in the 2D setting by \cite{KaPr:2008}, and for suitably
	regularized problem again in 3D setting \cite{PrZa:2013}.

	\section{\tcr{Formulation of the Problem and the Main Result.} }
	
	We consider generalized Navier-Stokes equations with dynamic
	boundary condition on a bounded domain $\Omega \subset \BR^3,\
	\Omega \in \calC^{0,1}$ and bounded time interval $(0,T)$. We
	denote space-time domain by $Q := (0,T) \times \Omega$, and by
	$\Gamma := (0,T) \times \partial \Omega$ the space-time boundary. 
	We further denote unknown velocity by $v: Q \rightarrow \BR^3$ 
	and unknown pressure of the fluid by $\pi:Q \rightarrow \BR$. 
	The quantity $\bs{S}$ is called the extra stress tensor and
	here it is assumed to be a function of the symmetric velocity
	gradient $2 \bs{Dv} = \nabla \bs{v} + (\nabla \bs{v})^T$. 
	The external body force $\bs{f}:Q \rightarrow \BR^3$ is
	independent of time.
	
	\par
	An essential feature of our model is that we incorporate
	the so-called dynamic boundary condition, so that the
	tangential velocity component is subject to a certain
	non-linear response $\bs{s}=\bs{s}(\bs{v})$ on $\Gamma$. 
	Our system thus reads
	\begin{subequations}\label{nse_L}
		\begin{align}
		\partial_t \bs{v} - \div \ \bs{S} + \div \ (\bs{v} \otimes \bs{v}) + \nabla \pi &= \bs{f} & \text{ in } Q, \label{nse_L-a}\\
		\div \ \bs{v} &= 0 & \text{ in } Q, \label{nse_L-b}\\
		\bs{v} \cdot \bs{n} &= 0 & \text{ on } \Gamma, \label{nse_L-c}\\
		-(\bs{Sn})_{\tau} &= \alpha \bs{s} + \beta \partial_t \bs{v} & \text{ on } \Gamma, \label{nse_L-d}\\
		\bs{v}(0) &= \bs{v}_0 & \text{ in } \Omega \cup \partial \Omega. \label{nse_L-e}
		\end{align}
	\end{subequations}
	Concerning the constitutive functions $\bs{S} = \bs{S}(\bs{D} \bs{v})$ 
	and $\bs{s}=\bs{s}(\bs{v})$, \tcb{we
	assume polynomial growth in terms of certain $r$ and $q\ge2$. More precisely:}
	for all $\ds \bs{D}_1, \bs{D}_2 \in \BR^{3 \times 3}_{sym}$
	
	\begin{equation}\label{prop_S}
	\begin{aligned}
	\bs{S}(\bs{0}) &= \bs{0},\\
	\abs{\bs{S}(\bs{D}_1) - \bs{S}(\bs{D}_2)} &\leq c_1 \left( \nu_1 + \nu_2 \left( \abs{\bs{D}_1} + \abs{\bs{D}_2} \right)^{r-2} \right) \abs{\bs{D}_1 - \bs{D}_2},\\
	\left(\bs{S}(\bs{D}_1) - \bs{S}(\bs{D}_2)\right):\left(\bs{D}_1 - \bs{D}_2\right) & \geq c_2 \left( \nu_1 + \nu_2 \left( \abs{\bs{D}_1} + \abs{\bs{D}_2} \right)^{r-2} \right) \abs{\bs{D}_1 - \bs{D}_2}^2.
	\end{aligned}	
	\end{equation}
	\nin Furthermore, it is assumed that $\bs{S}$ has a potential,
	\begin{equation}\label{prop_Phi}
	\begin{aligned}
	\bs{S}(\bs{D}) &= \partial_{\bs{D}} \Phi \left( \abs{\bs{D}}^2 \right),\\
	c_3 \left( \nu_1 + \nu_2 \abs{\bs{D}}^{r-2} \right)\abs{\bs{D}}^2 &\leq \Phi \left(\bs{D}\right) \leq c_4 \left( \nu_1 + \nu_2 \abs{\bs{D}}^{r-2} \right)\abs{\bs{D}}^2.
	\end{aligned}
	\end{equation}
	Typical example is the so-called Ladyzhenskaya fluid
	\begin{equation}
	\bs{S}(\bs{D}) = \nu_1 \bs{D} \bs{v} + \nu_2 \abs{\bs{Dv}}^{r-2} \bs{D} \bs{v} 
	\end{equation}
	Regarding the boundary nonlinearity $\bs{s}$, we require that for
	all $\bs{v}_1$, $\bs{v}_2 \in \BR^3$
	\begin{equation}\label{s-prop}
	\begin{aligned}
	\bs{s}(\bs{0}) &= \bs{0},	\\
	\abs{\bs{s}(\bs{v}_1) - \bs{s}(\bs{v}_2)} &\le
	c_5 \abs{\bs{v}_1 - \bs{v}_2}, \\
	\left( \bs{s}(\bs{v}_1) - \bs{s}(\bs{v}_2) \right) \cdot
	\left( \bs{v}_1 - \bs{v}_2 \right) 
	&\ge c_6 \abs{ \bs{v}_1 - \bs{v}_2 }^2,\\
	\tcb{\bs{s}(\bs{v})  \bs{v} } & \tcb{\geq c_7 \left(\abs{\bs{s}}^{\bar{q}} + \abs{\bs{v}}^q\right), \text{ where } \rfrac{1}{q} + \rfrac{1}{\bar q} = 1.}
	\end{aligned}
	\end{equation}
	Here, we also impose the existence of a potential, i.e.
	\begin{equation}	\label{pot-s}
	\bs{s}(\bs{v}) = \partial_{\bs{v}} \mathcal{S}(\bs{v})
	\end{equation}
	Without loss of generality, let $\mathcal{S}(\bs{0})=\bs{0}$.
	It is obvious that $\mathcal{S}$ obeys upper and lower
	\tcb{$q$-growth bounds}, in view of \eqref{s-prop}.\\
	\par
	
	\nin \tcr{Our main result, stated somewhat informally, reads as follows.}\\
	\par
	\nin \tcr{ {\bf Main Theorem.} {\it Let $r>12/5$ and $\bs{f} \in L^2(\Omega)$. Then 
			the system (\ref{nse_L-a} -- \ref{nse_L-e}) has a global attractor in $L^2(\Omega) \times L^2(\partial \Omega)$. Moreover, its dimension can be explicitly estimated in terms of the data.}}\\
	\par
	\tcr{	See Theorem~\ref{thm:main} below for a precise statement and proof. We note that the solutions are not uniquely determined by initial conditions in $L^2$ only. Yet they immediately become more regular (and hence unique), as follows from Theorems~\ref{thm:uniq} and \ref{thm:regdt}. This issue of initial nonuniqueness is easily avoided in our setting of short trajectories.}
	\par
	\tcr{As a by-product of the time regularity, we obtain that the attractor is bounded
		in $W^{1,r}$, and the solutions on attractor are $1/2$-H\"older continuous with
		values in $L^2$. One can expect that additional, i.e. spatial regularity is
		also available, so that the solutions would be in fact strong. We leave this
		problem to the forthcoming paper.}

	\section{\tcr{Well-posedness and Additional Time Regularity.}}
	
	We carry out our analysis with dynamical boundary condition which includes the time derivative of the velocity $\bs{v}$ of the fluid weighted by the parameter $\beta$. This set up demands a specific type function spaces. First we introduce such function spaces and later we define the Gelfand triplet. We essentially follow the functional set up used in \cite[Section 3]{ABM:2021}.\\
	
	\nin For $\Omega$ a Lipschitz domain in $\BR^d$, i.e., $\Omega \in \calC^{0,1}, \ \beta \geq 0$ and $r \in (0,\infty)$, we define $\calV \subset \calC^{0,1}(\overline{\Omega}) \times \calC^{0,1}(\partial \Omega)$ as
	\begin{equation*}
	\calV := \left\{ \left( \bs{v}, \bs{g} \right) \in \calC^{0,1}(\overline{\Omega}) \times \calC^{0,1}(\partial \Omega): \div \ \bs{v} = 0 \text{ in  } \Omega, \ \bs{v} \cdot \bs{n} = 0, \text{ and } \bs{v} = \bs{g} \text{ on } \partial \Omega \right\}
	\end{equation*}
	\nin With the help of $\calV$, we define
	\begin{align}
	V_r &:= \overline{\calV}^{\norm{\cdot}_{V_r}}, \ \text{where} \ \norm{(\bs{v}, \bs{g})}_{V_r} := \norm{\bs{v}}_{W^{1,r}(\Omega)} + \norm{\bs{v}}_{L^2(\Omega)} + \norm{\bs{g}}_{L^2(\partial \Omega)},\\
	H &:= \overline{\calV}^{\norm{\cdot}_{H}}, \ \text{where} \ \norm{(\bs{v}, \bs{g})}^2_{H} := \norm{\bs{v}}^2_{L^2(\Omega)} + \beta \norm{\bs{g}}^2_{L^2(\partial \Omega)}
	\end{align}
	\tcb{Note that $H$ is a Hilbert space with respect to the above norm.
	    We also remark that if $(\bs{v},\bs{g}) \in V_r$, then necessarily 
	    $\bs{g} = tr \ \bs{v}$. With some abuse of notation, $V_r$ can thus 
	    be identified with its first component $v$. }
	
	\begin{thm}\label{thm-wp}
		Let $\bs{v}_0 \in H$, $\bs{f} \in L^{r'}(0,T;V'_r)$, $T > 0$ be given, 
		and let $r \geq \rfrac{11}{5}$. 
		Then there exists at least one weak solution $\bs{v}$ to \eqref{nse_L},
		\begin{equation}
		\begin{aligned}
		\bs{v} &\in L^{\infty}(0,T;H) \cap L^r(0,T;V_r),\\
		\partial_t \bs{v} &\in L^{r'}(0,T;V'_r).
		\end{aligned}		
		\end{equation}
		The solution satisfies energy equality, and the initial
		condition $\bs{v}(0) = \bs{v}_0$ holds for the representative 
		$\bs{v} \in C([0,T];H)$.
	\end{thm}
	\begin{proof}
		We only sketch the proof, referring to \cite{ABM:2021} for details.
		Take the scalar product of \eqref{nse_L-a} with an arbitrary
		$\bs{\varphi} \in V_r$, integrate the result over $\Omega$, and use integration by parts to obtain
		\begin{equation}
		\int_{\Omega} [\partial_t \bs{v} \cdot \bs{\varphi} + (\bs{S} - \bs{v} \otimes \bs{v}): \nabla \bs{\varphi} - \pi \div  \bs{\varphi}] \ dx + \int_{\partial \Omega} \left[\tcb{\pi} \bs{I} + \bs{v} \otimes \bs{v} - \bs{S} \right] \bs{n} \cdot \bs{\varphi} \ dS = \int_{\Omega} \bs{f} \cdot \bs{\varphi} \ dx
		\end{equation}
		\nin By utilizing the symmetry of $\bs{S}$, \eqref{nse_L-c}, \eqref{nse_L-d}, and the properties of $\bs{\varphi}$ ($\div \bs{\varphi} = 0$ in $\Omega$, $\bs{\varphi} \cdot \bs{n} = 0$ on $\partial \Omega$), we deduce the weak formulation
		\begin{multline}\label{wk-form}
		\int_{\Omega} \partial_t \bs{v} \cdot \bs{\varphi} \ dx + \beta \int_{\partial \Omega} \partial_t \bs{v} \cdot \bs{\varphi} \ dS  + \int_{\Omega} \left[ \tcb{\bs{S} (\bs{Dv})} - \bs{v} \otimes \bs{v} \right]: \nabla \bs{\varphi} \ dx + \alpha \int_{\partial \Omega} \tcr{\bs{s}}(\bs{v}) \cdot \bs{\varphi} \ dS \\ = \int_{\Omega} \bs{f} \cdot \bs{\varphi} \ dx
		\end{multline}
		\nin Formally, we set $\bs{\varphi} := \bs{v}$ in \eqref{wk-form}, and use 
		\begin{align}
		\int_{\Omega} (\bs{v} \otimes \bs{v}): \nabla \bs{v} dx &= \int_{\Omega} \sum_{i,j = 1}^d\bs{v}_i\bs{v}_j \partial_i\bs{v}_jdx = \frac{1}{2} \int_{\Omega} \sum_{i,j = 1}^d \bs{v}_i \partial_i\abs{\bs{v_j}}^2dx,\\
		&= \frac{1}{2} \left( -\int_{\Omega} \div \bs{v} \abs{\bs{v}}^2 dx + \int_{\partial \Omega} \bs{v} \cdot \bs{n} \abs{\bs{v}}^2 dS \right) = 0,
		\end{align}
		where we have used \eqref{nse_L-b}, \eqref{nse_L-c}. Thus we obtain,
		\begin{equation} 	\label{ener-eq}
		\frac{1}{2}\tcr{\frac{d}{dt}}\left(\int_{\Omega}  \abs{\bs{v}}^2 \ dx + \beta \int_{\partial \Omega} \abs{\bs{v}}^2 \ dS \right)  
		+ \int_{\Omega} \tcb{\bs{S}(\bs{Dv}):\bs{Dv}}  \ dx + \alpha \int_{\partial \Omega} \tcr{\bs{s}}(\bs{v}) \cdot \bs{v} \ dS = \tcb{\ip{\bs{f}}{\bs{v}}_{V_r',V_r}}. 
		\end{equation}	
		\nin \tcb{For the right hand side of \eqref{ener-eq}, we obtain by utilizing Korn's and Young's inequalities, 
		\begin{align*}
		\ip{\bs{f}}{\bs{v}}_{V_r',V_r} \leq \norm{\bs{f}}_{V_r'} \norm{\bs{v}}_{V_r} \leq c_1 \norm{\bs{f}}_{V_r'} \norm{\bs{v}}_{W^{1,r}} &\leq c_2(\varepsilon) \norm{\bs{f}}^{r'}_{V_r'} + \frac{\varepsilon}{c_3}\norm{\bs{v}}^r_{W^{1,r}}, \ (\varepsilon > 0)\\
		&\leq c_2(\varepsilon) \norm{\bs{f}}^{r'}_{V_r'} + \varepsilon\norm{\bs{Dv}}^r_r + \varepsilon\norm{\bs{v}}^r_{H}.
		\end{align*}}
		Then by \tcb{\eqref{prop_S} and} \eqref{s-prop}, we deduce,
		\begin{equation}
		\frac{1}{2} \frac{d}{dt} \norm{\bs{v}}^2_{H} + c_5 \left[\nu_1 \norm{\bs{Dv}}^2_2 + \nu_2 \norm{\bs{Dv}}^r_r\right] \tcb{ + c_4\alpha \norm{\bs{v}}_{L^q(\Gamma)}^q} \leq c_2(\varepsilon) \norm{\bs{f}}^{r'}_{V_r'} + \varepsilon\norm{\bs{v}}^r_{H}.
		\end{equation}
		We combine compactness and monotonicity arguments to obtain the
		existence of a solution as a limit of a suitable approximate
		problem,  e.g. the Galerkin scheme. Remark that $r = \rfrac{11}{5}$ is the
		critical value which ensures that the convective term belongs
		to the proper dual space. Hence in particular, any weak
		solution is an admissible test function and the energy equality
		\eqref{ener-eq} holds. See \cite{Mar:2019} or \cite{ABM:2021}.
	\end{proof}
	
	Weak solutions are non-unique in general, unless additional
	regularity is assumed. In particular, analogously to 
	\cite[Theorem 3.2]{BEKP:2009}, one proves:
	\begin{thm}	\label{thm:uniq}
		Let $\bs{u}, \bs{v}$ be weak solutions with $\bs{u}(0) = \bs{v}(0)$, and furthermore, let $\ds \bs{v} \in L^{\frac{2r}{2r - 3}}(0,T;V_r)$. Then $\bs{u} = \bs{v}$. 
	\end{thm}
	\begin{proof}
		Test the equation for $\bs{w} := \bs{u} - \bs{v}$ by $\bs{w}$. 
		Using the identity
		\begin{equation}
		\int_{\Omega} \left( \bs{u} \otimes \bs{u} - \bs{v} \otimes \bs{v} \right):\nabla \bs{w} = \int_{\Omega} \left( \bs{u} \otimes \bs{w} - \bs{w} \otimes \bs{v} \right):\nabla \bs{w} = \int_{\Omega} \left( \bs{w} \cdot \nabla \bs{v} \right) \cdot \bs{w}
		\end{equation}
		(\tcb{in view of} $\ds \div \left( \bs{u} \otimes \bs{u} \right) = \left(\nabla \cdot \bs{u}\right) \bs{u} + \left(\bs{u} \cdot \nabla \right) \bs{u} $) as well as \eqref{prop_S}, one obtains
		\begin{equation}
		\frac{1}{2} \frac{d}{dt} \left( \norm{\bs{w}}^2_2 + \beta \norm{\bs{w}}^2_{L^2(\partial \Omega)} \right) + c_2 \int_{\Omega} I^2 (\bs{Du}, \bs{Dv}) \ dx + \alpha \int_{\partial \Omega} \left( \tcr{\bs{s}}(\bs{u}) - \tcr{\bs{s}}(\bs{v}) \right) \bs{w} \,\tcr{dS} \leq \int_{\Omega} \abs{\bs{w}}^2 \abs{\nabla \bs{v}}\tcr{dx}, 
		\end{equation}
		\nin where 
		\begin{equation}\label{def-IDD}
		I^2 (\bs{Du}, \bs{Dv}) := \left( \nu_1 + \nu_2 \left( \abs{\bs{Du}} + \abs{\bs{Dv}} \right)^{r-2} \right)\abs{\bs{Dw}}^2
		\end{equation}
		\nin By monotonicity we have 
		\begin{equation*}
		\alpha \int_{\partial \Omega} \left( \tcr{\bs{s}}(\bs{u}) - \tcr{\bs{s}}(\bs{v}) \right) \bs{w} \,\tcr{dS} \geq 0.
		\end{equation*}
		This yields
		\begin{equation}
		\frac{1}{2} \frac{d}{dt} \left( \norm{\bs{w}}^2_2 + \beta \norm{\bs{w}}^2_{L^2(\partial \Omega)} \right) + c_2 \int_{\Omega} I^2 (\bs{Du}, \bs{Dv}) \ dx \leq \int_{\Omega} \abs{\bs{w}}^2 \abs{\nabla \bs{v}} \tcr{dx} 
		\end{equation}
		\nin By Korn inequality (Lemma~\ref{lem-Korn} in the Appendix), we have
		\begin{equation}\label{idd-lb}
		\int_{\Omega} I^2 (\bs{Du}, \bs{Dv}) \ dx \geq \nu_1 \int_{\Omega} \abs{\bs{Dw}}^2\tcr{dx} + \nu_2 \int_{\Omega} \abs{\bs{Dw}}^r \tcr{dx} \geq c \nu_1 \left(\norm{\bs{w}}^2_{W^{1,2}(\Omega)} - \norm{\bs{w}}^2_{L^2(\partial \Omega)} \right).
		\end{equation}
		\nin We further estimate, using \eqref{int-inq}, cf.
		the Appendix, 
		\begin{align*}
		\int_{\Omega} \abs{\bs{w}}^2 \abs{\nabla \bs{v}}\tcr{dx} &\leq \norm{\nabla \bs{v}}_r \norm{\bs{w}}^2_{\frac{2r}{r-1}} \leq c_3 \norm{\nabla \bs{v}}_r \norm{\bs{w}}^{\frac{2r-3}{r}}_2 \norm{\bs{w}}^{\frac{3}{r}}_{W^{1,2}(\Omega)},\\			
		&\leq \frac{c_2}{4} \nu_1 \norm{\bs{w}}^2_{W^{1,2}(\Omega)} + c_4 \nu_1^{-\frac{3}{2r-3}} \norm{\nabla \bs{v}}^{\frac{3}{2r-3}}_r \norm{\bs{w}}_2^2.
		\end{align*}
		\nin Then with \eqref{idd-lb} we obtain,
		\begin{align}
		\frac{d}{dt} \norm{\bs{w}}^2_H + c_5 \nu_1 \norm{\bs{w}}^2_{W^{1,2}(\Omega)} + c_5 \int_{\Omega} &I^2 (\bs{Du}, \bs{Dv}) \ dx \nonumber\\
		&\leq c_4 \nu_1^{-\frac{3}{2r-3}} \norm{\bs{v}}^{\frac{3}{2r-3}}_{W^{1,r}(\Omega)} \norm{\bs{w}}_2^2 + c_6 \norm{\bs{w}}^2_{L^2(\partial \Omega)}, \nonumber\\
		& \leq c_4 \nu_1^{-\frac{3}{2r-3}} \norm{\bs{v}}^{\frac{3}{2r-3}}_{W^{1,r}(\Omega)} \norm{\bs{w}}_2^2 + c_6 \norm{\bs{w}}^2_{L^2(\partial \Omega)}, \nonumber\\ 
		& \leq c_7 \left(\nu_1^{-\frac{3}{2r-3}}\norm{\bs{v}}^{\frac{3}{2r-3}}_{W^{1,r}(\Omega)}  + 1 \right) \norm{\bs{w}}^2_H.  \label{de-est}
		\end{align}
		\nin \tcb{Finally} we apply Gr\"{o}nwall's lemma to deduce
		\begin{equation}
		\norm{\bs{w}(t)}^2_2 \leq K \norm{\bs{w}(s)}^2_2, \quad 0 \leq s \leq t \leq T.
		\end{equation}
		In particular, we have uniqueness. 
	\end{proof}
	\tcb{Now}, we obtain additional time regularity of the solutions,
	together with an explicit estimate of the relevant norms,
	cf. \cite[Theorem 3.3]{BEKP:2009}.
	\tcb{Symbol $``\lesssim"$ means an inequality up to
	some generic (i.e., independent of the data) constant $c_i > 0$. }
	\begin{thm} \label{thm:regdt}
		Let $r > 12/5$, $\bs{f} \in \tcb{L^2(\Omega)}$. 
		Then the weak solution has additional time regularity
		\begin{align*}
		\bs{v} \in L^{\infty}(\tau, T; V_r), \\
		\partial_t \bs{v} \in L^2(\tau, T; L^2(\Omega)).
		\end{align*}
		Here $\tau \in (0,T)$ is arbitrary, and one can take $\tau = 0$ if $\bs{v}(0) \in V_r$.
	\end{thm}
	\nin Now let $\bs{\varphi} := \partial_t \bs{v}$
	\begin{multline*}
	\int_{\Omega} \abs{\partial_t \bs{v}}^2 \ dx + \beta \int_{\partial \Omega} \abs{\partial_t \bs{v}}^2 \ dS  + \int_{\Omega} \left[ \nu_1 \bs{D} \bs{v} + \nu_2 \abs{\bs{Dv}}^{r-2} \bs{D} \bs{v} - \bs{v} \otimes \bs{v} \right]: \nabla \partial_t \bs{v} \ dx \\ + \alpha \int_{\partial \Omega} \tcr{\bs{s}}(\bs{v}) \cdot \partial_t \bs{v} \ dS = \int_{\Omega} \bs{f} \cdot \partial_t \bs{v} \ dx\\
	\norm{\partial_t \bs{v}}^2_H  + \int_{\Omega} \left[ \nu_1 \bs{D} \bs{v} + \nu_2 \abs{\bs{Dv}}^{r-2} \bs{D} \bs{v} \right]: \partial_t \bs{Dv} \ dx + \int_{\Omega} (\bs{v} \cdot \nabla \bs{v}) \cdot \partial_t \bs{v} \,\tcr{dx} + \alpha \int_{\partial \Omega} \tcr{\bs{s}}(\bs{v}) \cdot \partial_t \bs{v} \ dS \\= \int_{\Omega} \bs{f} \cdot \partial_t \bs{v} \ dx
	\end{multline*}
	\nin We estimate, 
	\begin{equation}
	\int_{\Omega} (\bs{v} \cdot \nabla \bs{v}) \cdot \partial_t \bs{v} \ \tcr{dx} \leq \norm{\bs{v}}^2_{\frac{2r}{r-2}}\norm{\bs{v}}^2_{W^{1,r}(\Omega)} + \frac{1}{4}\norm{\partial_t\bs{v}}^2_2.
	\end{equation}
	\nin Now by \eqref{pot-s}, we obtain,
	\begin{equation*}
	\int_{\partial \Omega} \tcr{\bs{s}}(\bs{v}) \cdot \partial_t \bs{v} \ dS = \frac{d}{dt} \left(\int_{\partial \Omega} \bs{\calS} (\bs{v}) \ dS\right).
	\end{equation*}
	
	\nin Then we obtain the following inequality,
	\begin{equation}
	\frac{1}{2} \norm{\partial_t \bs{v}}^2_H  + 
	\tcb{ \frac{d}{dt} \int_{\Omega} \Phi(\bs{Dv}) \ dx }
	+ \frac{d}{dt} \left(\int_{\partial \Omega} \bs{\calS} (\bs{v}) \ dS\right) \leq  c_8 \norm{\bs{v}}^2_{\frac{2r}{r-2}}\norm{\bs{v}}^2_{W^{1,r}(\Omega)} + \norm{\bs{f}}_2^2
	\end{equation}
	\nin \tcb{This can be more compactly written as}
	\begin{equation}\label{ine_dU}
	\frac{1}{2} \norm{\partial_t \bs{v}}^2_H  + \frac{d}{dt}U \\ \leq c_8 \norm{\bs{v}}^2_{\frac{2r}{r-2}}\norm{\bs{v}}^2_{W^{1,r}(\Omega)} + \norm{\bs{f}}_2^2,
	\end{equation}
	\nin where 
	\begin{multline}
	U = U(t) := 1 + \int_{\Omega} \Phi(\bs{Dv}) \ dx + \int_{\partial \Omega} \bs{\calS} (\bs{v}) \ dS, \\ \text{and hence} \ U \sim 1 + \nu_1 \norm{\bs{Dv}}_2^2 + \nu_2 \norm{\bs{Dv}}_r^r + \norm{\bs{\calS} (\bs{v})}_{L^1(\partial \Omega)},
	\end{multline}
	\nin by \eqref{prop_Phi} and Korn's inequality \eqref{lem-Korn}. Now we distinguish two cases:
	\begin{enumerate}[(i)]
		\item case $r \in (12/5, 3]$. We claim
		\begin{equation}\label{bd-2r-3}
		\norm{\bs{v}}_{\frac{2r}{r-2}} \leq c_9 \norm{\bs{v}}^a_2 \norm{\bs{v}}^{1-a}_{W^{1,r}(\Omega)}, \quad a = \frac{5r - 12}{5r - 6}.
		\end{equation}
		\nin Note that $a > 0$ as $r > 12/5$. Then for $r \leq 3$, the embedding $\ds W^{1,r}(\Omega) \subset L^{\frac{3r}{3-r}}(\Omega)$ holds. We obtain
		\begin{equation}
		\norm{\bs{v}}_{\frac{2r}{r-2}} \leq \norm{\bs{v}}^a_2 \norm{\bs{v}}^{1-a}_{{\frac{3r}{3-r}}} \leq c_9 \norm{\bs{v}}^a_2 \norm{\bs{v}}^{1-a}_{W^{1,r}(\Omega)}.
		\end{equation}
		\nin Then we estimate the first term on the right hand side of \eqref{ine_dU} and obtain
		\begin{align*}
		\norm{\bs{v}}^{\frac{2(5r - 12)}{5r - 6}}_2 \norm{\bs{v}}^{\frac{10r}{5r - 6}}_{W^{1,r}(\Omega)} &\leq \nu_2^{-\frac{10}{5r-6}} \norm{\bs{v}}^{\frac{2(5r - 12)}{5r - 6}}_2 \left[ \nu_2 \norm{\bs{v}}^r_{W^{1,r}(\Omega)}\right]^{\frac{10}{5r - 6}},\\
		&\lesssim \nu_2^{-\frac{10}{5r-6}} \norm{\bs{v}}^{\frac{2(5r - 12)}{5r - 6}}_2 \left[ \nu_2\norm{\bs{Dv}}^r_r + \nu_2\norm{\bs{v}}^r_2 \right]^{\frac{10}{5r - 6}},\\
		&\lesssim \nu_2^{-\frac{10}{5r-6}} \norm{\bs{v}}^{\frac{2(5r - 12)}{5r - 6}}_2 U^{\frac{10}{5r - 6}} + \norm{\bs{v}}^4_2. 
		\end{align*}
		\nin This yields
		\begin{equation}
		\frac{d}{dt} U \leq c_{10} \nu_2^{-\frac{10}{5r-6}} \norm{\bs{v}}^{\frac{2(5r - 12)}{5r - 6}}_2  U^{\frac{10}{5r - 6}} + \norm{\bs{v}}^4_2 + \norm{\bs{f}}_2^2.
		\end{equation}
		\nin Dividing by $U^{1-\mu}$, where $\mu = \frac{2(5r-12)}{5r-6}$ yields,
		\begin{equation}\label{U-2.4}
		\frac{d}{dt} U^{\mu} \leq c_{10} \nu_2^{-\frac{10}{5r-6}} \norm{\bs{v}}^{\frac{2(5r - 12)}{5r - 6}}_2  U + \norm{\bs{v}}^4_2 + \norm{\bs{f}}_2^2. 
		\end{equation}		
		\nin Then we apply \tcb{Gr\"{o}nwall's} lemma to obtain the necessary bounds on $U$. It is worthwhile to note that
		\begin{equation*}
		\left[ \norm{\bs{v}}^4_2 + \norm{\bs{f}}_2^2 \right]U^{\mu - 1} = \left[ \norm{\bs{v}}^4_2 + \norm{\bs{f}}_2^2 \right] U^{\frac{5r - 16}{5r - 6}} \leq \norm{\bs{v}}^4_2 + \norm{\bs{f}}_2^2.
		\end{equation*}
		The above property holds true because $U \geq 1$ and for $r \leq 3$, we have $\frac{5r - 16}{5r - 6} < 0$.
		\item case $r > 3$. 
		\nin Since $\ds \frac{2r}{r-2} \in (2,6)$, we use the interpolation Lemma \ref{lem-int} to obtain,
		\begin{equation}
		\norm{\bs{v}}_{\frac{2r}{r-2}} \lesssim \norm{\bs{v}}_2^{\frac{r-3}{r}} \norm{\bs{v}}_6^{\frac{3}{r}}.
		\end{equation}
		\nin Again by Lemma \ref{lem-int-w} we obtain,
		\begin{equation}\label{bd-2r-2}
		\norm{\bs{v}}_{\frac{2r}{r-2}} \lesssim \norm{\bs{v}}_2^{\frac{r-2}{r}} \norm{\bs{v}}_{W^{1,3}(\Omega)}^{\frac{2}{r}} \leq c_{11} \norm{\bs{v}}_2^{\frac{r-2}{r}} \norm{\bs{v}}_{W^{1,r}(\Omega)}^{\frac{2}{r}}.
		\end{equation}
		\nin Then right hand side of \eqref{ine_dU} can be estimated as
		\begin{align}
		\norm{\bs{v}}_2^{\frac{2(r-2)}{r}} \norm{\bs{v}}_{W^{1,r}(\Omega)}^{\frac{2r + 4}{r}} &\leq \norm{\bs{v}}_2^{\frac{2(r-2)}{r}} \left[ \norm{\bs{Dv}}^r_r + \norm{\bs{v}}^r_2 \right]^{\frac{2r + 4}{r^2}},\nonumber \\
		&\leq \nu_2^{-\frac{r^2}{2r + 4}} \norm{\bs{v}}_2^{\frac{2(r-2)}{r}} U^{\frac{2r + 4}{r^2}} + \norm{\bs{v}}^4_2.
		\end{align}
		\nin This yields
		\begin{equation}\label{dUdt}
		\frac{d}{dt}U \leq c_{12} \nu_2^{-\frac{r^2}{2r + 4}} \norm{\bs{v}}_2^{\frac{2(r-2)}{r}} U^{\frac{2r + 4}{r^2}} + \norm{\bs{v}}^4_2 + \norm{\bs{f}}_2^2.
		\end{equation}
		\nin Take $\ds \mu = \frac{2r + 4}{r^2} - 1$. Then we consider two cases. \\
		
		\nin If $\ds \frac{2r + 4}{r^2} > 1$, i.e. $\mu > 0$, we divide \eqref{dUdt} by $U^{\mu}$. Thus we obtain
		\begin{equation*}
		\frac{d}{dt}U^{1 - \mu} \leq c_{12} \nu_2^{-\frac{r^2}{2r + 4}} \norm{\bs{v}}_2^{\frac{2(r-2)}{r}} U + \norm{\bs{v}}^4_2 + \norm{\bs{f}}_2^2.
		\end{equation*}
		\nin Similar to the previous case where $r \in (12/5,3]$, we observe that
		\begin{equation*}
		\left[ \norm{\bs{v}}^4_2 + \norm{\bs{f}}_2^2 \right] U^{\mu} \leq \norm{\bs{v}}^4_2 + \norm{\bs{f}}_2^2.
		\end{equation*}
		\nin If $\ds \frac{2r + 4}{r^2} \leq 1$,  i.e. $\mu \leq 0$, we obtain by \eqref{dUdt},
		\begin{equation}
		\frac{d}{dt}U \leq c_{12} \nu_2^{-\frac{r^2}{2r + 4}} \norm{\bs{v}}_2^{\frac{2(r-2)}{r}} U + \norm{\bs{v}}^4_2 + \norm{\bs{f}}_2^2.
		\end{equation}
		\nin Then in both cases, we invoke \tcb{Gr\"{o}nwall's} lemma to obtain bounds on $U$.
	\end{enumerate}
	
	\section{Dimension of the Attractor}
	We follow the general scheme of method of trajectories presented in \cite{MP:2002}. The main \tcb{modification} here is that we explicitly keep track of all a priori estimates.
\tcb{	
	\begin{lemma}
		There exists an absorbing, positively invariant set $\hat{\calB} \subset H$ such that 
		\begin{equation}
		B_0 := \sup_{\bs{v} \in \hat{\calB}} \norm{\bs{v}}_H \leq c_1 \min \left\{ \kappa_1^{-1}\norm{\bs{f}}_2, \left[\kappa_2^{-1}\norm{\bs{f}}_2\right]^{\frac{1}{s-1}} \right\},
		\end{equation}
		where $s = \min\{ r, q \}$, and $\kappa = \min\{ \nu_2, \alpha \}$.
	\end{lemma}
	\begin{proof}
		As in Theorem \ref{thm-wp}, we obtain
		\begin{equation}
		\frac{d}{dt} \norm{\bs{v}}^2_{H} + c_2 \left[\nu_1 \norm{\bs{Dv}}^2_2 + \nu_2 \norm{\bs{Dv}}^r_r + \alpha\norm{\bs{v}}^q_{L^q(\Gamma)} \right] \leq c_3 \norm{\bs{f}}_2\norm{\bs{v}}_H.
		\end{equation}
		\nin Then by dropping the term $\norm{\bs{Dv}}^r_r$, we compute by Korn's inequality in Lemma \ref{lem-Korn},
		\begin{equation}
		\frac{d}{dt} \norm{\bs{v}}^2_{H} + c_2 \kappa_1 \norm{\bs{v}}^2_H \leq c_3 \norm{\bs{f}}_2\norm{\bs{v}}_H, \ \text{where } \kappa_1 = \min\{\nu_1, \alpha \} \label{ine_wl}
		\end{equation}
		\nin Thus $\ds \frac{d}{dt} \norm{\bs{v}}^2_{H} \le -\gamma \norm{\bs{v}}^2_H$ if $\norm{\bs{v}}_H > c_4 \kappa_1^{-1} \norm{\bs{f}}_2$ for some $\gamma  > 0$. Now we  drop the term $\norm{\bs{Dv}}^2_2$ and obtain,
		\begin{equation}
		\frac{d}{dt} \norm{\bs{v}}^2_{H} + c_2 \left[\nu_2 \norm{ \bs{Dv}}^r_r + \alpha \norm{\bs{v}}^q_{L^q(\Gamma)} \right] \leq c_3 \norm{\bs{f}}_2\norm{\bs{v}}_H. 
		\end{equation}
		Then we use the following estimate for $\norm{\bs{v}}_H \geq 1$,
		\begin{equation*}
			\nu_2 \norm{\bs{Dv}}^r_r + \alpha \norm{\bs{v}}^q_{L^q(\Gamma)} \geq c_5 \kappa \norm{\bs{v}}^s_H, \text{ where } s = \min\{ r, q \}, \text{ and } \kappa_2 = \min\{ \nu_2, \alpha \}.
		\end{equation*}
		Thus we obtain $\ds \frac{d}{dt} \norm{\bs{v}}^2_{H} \le -\gamma \norm{\bs{v}}_H^2$ if $\ds \norm{\bs{v}}_H > c_6  \left[\kappa^{-1}\norm{\bs{f}}_2\right]^{\frac{1}{s-1}}$ for some $\gamma  > 0$. Hence the conclusion follows.
	\end{proof}
}

	\begin{lemma}\label{lem-Br}
		There exists an absorbing, positively invariant $\calB \subset \hat{\calB}$ such that $\calB$ is closed in $H$, and 
		\begin{equation}\label{B_r}
		B_r := \sup_{\bs{v} \in \hat{\calB}} \norm{\bs{v}}_{W^{1,r}} \leq \begin{cases}
		c_{12} B^{\frac{5(5r-6)}{2(5r-11)}}_0, \ r \in (12/5,3],\\
		c_{12} B^5_0, \ r > 3.
		\end{cases}
		\end{equation}
	\end{lemma}
	\begin{proof}
		Set 
		\begin{equation*}
		\calB := \left\{ \bs{v}(2T); \ \bs{v} \text{ is a weak solution on } [0,2T], \text{and } \bs{v}(0) \in \hat{\calB}  \right\},
		\end{equation*}
		\nin and we take $T = B_0$. Recalling \eqref{ine_wl} and taking $U = \nu_1 \norm{\bs{Dv}}^2_2 + \nu_2 \norm{\bs{Dv}}^r_r$ 
		\begin{align}
		\int_{0}^{\tcr{T}} U(t) \ dt &\lesssim \int_{0}^{\tcr{T}} \left[\norm{\bs{f}}_2\norm{\bs{v}}_H + \frac{d}{dt} \norm{\bs{v}}^2_{H} \right] \ dt, \nonumber\\
		&\lesssim B_0^2 + TB_0 \leq 2c_1 B_0^2.
		\end{align}
		\nin By the mean value theorem of integrals, we obtain for $\tau \in (0,T)$ such that
		\begin{equation}
		U(\tau) \leq c_2 B_0.
		\end{equation} 
		\nin Assume $r \leq 3$. Integrating \eqref{U-2.4} over $(\tau, 2T)$ yields
		\begin{align*}
		U^{\mu}(2T) &\leq c_3 \nu_2^{-\frac{10}{5r-6}} B_0^{\frac{2(5r - 12)}{5r - 6}} \int_{\tau}^{2T} U(t) \ dt + \int_{\tau}^{2T} \left[ \norm{\bs{v}}^4_2 + \norm{\bs{f}}_2^2\right] \ dt + U^{\mu}(\tau),\\
		&\leq c_4 \nu_2^{-\frac{10}{5r-6}} B_0^{\frac{4(5r - 9)}{5r - 6}} + c_5B_0^{\mu} +  c_6B_0^5 + c_7B_0\norm{\bs{f}}_2^2.
		\end{align*}
		\nin Here $\mu = \frac{2(5r - 11)}{5r - 6}$. It is reasonable to assume that $B_0 > 1$, and $\nu_1, \nu_2 < 1$, hence the largest term is $B_0^{5/\mu}$-term. The above estimate only gives an upper bound for $\norm{\bs{Dv}}_r$. But by adding $\norm{\bs{v}}_2$ to both sides we obtain an upper bound for $\norm{\bs{v}}_{W^{1,r}}$. Then the desired estimate for $B_r$ holds.\\  
		
		\nin Then we compute for $r > 3$. Integrating 
		\begin{align*}
		U(2T) &\leq c_7 \nu_2^{-\frac{r^2}{2r + 4}} B_0^{\frac{2(r-2)}{r}} \int_{\tau}^{2T} U(t) \ dt + \int_{\tau}^{2T}\left[ \norm{\bs{v}}^4_2 + \norm{\bs{f}}_2^2\right]dt + U(\tau),\\
		&\leq c_8 \nu_2^{-\frac{r^2}{2r + 4}} B_0^{\frac{4(r-1)}{r}} + c_9 B_0^5 + c_{10}B_0\norm{\bs{f}}_2^2 + c_{11}B_0.
		\end{align*}
		\nin The largest term is the $B_0^5$-term. Hence the estimate follows. The closedness of $\calB$ follows from the compactness of the set of weak solutions, which is part of the existence theory. See the reference for Theorem~\ref{thm-wp}. \end{proof}
	
	\subsection{Attractors and Method of Trajectories.}
	Observe that by Theorems~\ref{thm:uniq}, \ref{thm:regdt},
	the solution operator $S(t):\bs{v}_0 \to \bs{v}(t)$ is
	well-defined for $\bs{v}_0 \in \calB$. It follows that 
	\begin{equation}
	\calA = \omega(\calB) = \bigcap_{\tau \ge 0 } \bigcup_{t\ge \tau}
	\overline{ S(t)\calB }^{H}
	\end{equation}
	is the so-called \emph{global attractor}. Our ultimate goal is
	to estimate its fractal dimension, defined as
	\begin{equation}
	d^H_f(\calA) = \limsup_{\varepsilon \to 0+}
	\frac{ \ln N_H(\calA,\varepsilon) }{ - \ln \varepsilon }
	\end{equation}
	where $N_H(\calA,\varepsilon)$ is the smallest number of
	$\varepsilon$-balls in the space $H$ that cover $\calA$. We employ the
	method of trajectories. Since the argument is very similar to
	\cite{BEKP:2009}, we only briefly sketch the main points. 
	We refer to \cite{MP:2002} for a more
	detailed description of the method; see also the introduction
	for other related references.
	\par
	Let $\ell>0$ be fixed; the exact value will specified in \eqref{val-l} below.
	The space of trajectories is defined as
	\begin{equation}	\label{def-bl}
	\Bell 
	= \left\{ \chi \in H_\ell;
	\textrm{ $\chi$ is a weak solution on $[0,\ell]$, $\chi(0)
		\in \calB$} \right\} \,,
	\end{equation}
	with the underlying metric of $H_\ell = L^2(0,\ell;H)$. Note
	however that any trajectory $\chi$ has additional regularity, cf.\
	Theorem ~\ref{thm-wp}.  In particular, we always work with the
	representative $\chi \in C([0,\ell];H)$, so that the value
	$\chi(t)$ is well-defined for any $t\in[0,\ell]$.
	The operators $\calL:\Bell \to \Bell$, $b:\calB \to \Bell$ and
	$e:\Bell \to \calB$ are defined via the conditions
	\begin{align*}
	\calL(\chi) &= \psi \iff \chi(\ell) = \psi(0)  \,, \\
	e(\chi) &= \chi(\ell) \,, \\
	b(\bs{v}_0) &= \chi \iff \chi(0) = \bs{v}_0 \,.
	\end{align*}
	Observe that $S(\ell) = e \circ b$ and $b \circ e =
	\calL$, hence $\calL$ is an equivalent (discrete) description
	of the dynamics of $S(t)$ on $\Bell = b(\calB)$. In particular,
	one has $\All = b(\calA)$, $\calA = e(\All)$, where $\All$ is
	the global attractor for the dynamical system
	$(\calL^n,\Bell)$.
	\par
	In view of the Lipschitz continuity of operators $e$, $b$
	(see for example \cite[Lemma 2.1]{MP:2002}, \cite[Lemma 1.2]{MP:2002})
	\begin{equation}
	d_f^H(\calA) = d_f^{L^2(0,\ell;H)}(\All)
	\end{equation}
	Thus, it suffices to estimate the last quantity. This will be
	done using the so-called smoothing property, see \cite[Lemma
	1.3]{MP:2002}; see also \cite[Theorem 4.1]{BEKP:2009}.
	It remains to explicitly estimate the appropriate Lipschitz
	constants, which is done in the following lemma. Finally, the
	asymptotics of covering numbers is investigated in the
	Appendix.
	
	\begin{lemma}\label{lem-L}
		Set 
		\begin{equation}\label{val-l}
		\ell := \left[\nu_1^{-\frac{3}{2r-3}}B^{\frac{2r}{2r-3}}_r +  1\right]^{-1}.
		\end{equation}
		\nin Then for all $\chi, \psi \in \calA_{\ell}$
		\begin{align}
		\norm{\calL \chi - \calL \psi}_{L^2(0,\ell;V_2)} \leq L_1 \norm{\chi - \psi}_{H_{\ell}}, \label{lip-L}\\
		\norm{\partial_t \calL \chi - \partial_t \calL \psi}_{L^2(0,\ell;V'_r)} \leq L_2 \norm{\chi - \psi}_{H_{\ell}}\label{lip-Lt}, 
		\end{align}
		\nin where
		\begin{align}
		L_1 &= c_1 \tcr{\nu_1}^{-\frac{1}{2}}\ell^{-\frac{1}{2}}, \label{L_1}\\
		L_2 &= U + W \tcr{+ Q}, \label{L_2}\\
		U &= c_2 \nu_1 L_1(1 + M_r), \label{U}\\
		M_r &= \nu_1^{-\frac{1}{2}} \nu_2^{\frac{1}{2}}B^{\frac{r-2}{2}}_r, \label{M_r}\\
		W &= \begin{cases}\label{W}
		c_4B_0^{\frac{5r-12}{5r-6}}B_r^{\frac{6}{5r-6}}, \ r \in (12/5, 3],\\
		c_4B_0^{\frac{r-2}{r}}B_r^{\frac{2}{r}}, \ r  > 3.
		\end{cases}
		\end{align}
	\end{lemma}
	\begin{proof}
		Let $\bs{u}, \bs{v}$ be two weak solutions on $[0,2\ell]$ such that $\ds \bs{u}|_{[0,\ell]} = \chi, \bs{u}|_{[0,\ell]} = \psi$, and set $\bs{w} := \bs{u} - \bs{v}$. In view of \eqref{val-l}, \eqref{de-est} is rewritten as 
		\begin{equation*}
		\frac{d}{dt} \norm{\bs{w}}^2_H + c_5 \nu_1 \norm{\bs{w}}^2_{W^{1,2}(\Omega)} + c_5 \int_{\Omega} I^2 (\bs{Du}, \bs{Dv}) \ dx \leq c_7 \left[\nu_1^{-\frac{3}{2r-3}}B^{\frac{2r}{2r-3}}_r + 1\right] \norm{\bs{w}}_H^2.
		\end{equation*}
		\nin We replace norm of the second term of the left hand side with the equivalent norm $\norm{ \ }_{V_2}$.
		This yields,
		\begin{equation*}
		\frac{d}{dt} \norm{\bs{w}}^2_H + c_8 \nu_1 \norm{\bs{w}}^2_{V_2} + c_8 \int_{\Omega} I^2 (\bs{Du}, \bs{Dv}) \ dx \leq c_7 \left[\nu_1^{-\frac{3}{2r-3}}B^{\frac{2r}{2r-3}}_r + 1\right] \norm{\bs{w}}_H^2.
		\end{equation*}
		\nin Then by \eqref{val-l} we obtain,
		\begin{equation}\label{de-est-1}
		\frac{d}{dt} \norm{\bs{w}}^2_H + c_8 \nu_1 \norm{\bs{w}}^2_{V_2} + c_8 \int_{\Omega} I^2 (\bs{Du}, \bs{Dv}) \ dx \leq c_7 \ell^{-1} \norm{\bs{w}}_H^2.
		\end{equation}
		Neglecting the positive terms of the \tcb{left hand side}, we obtain from \tcb{Gr\"{o}nwall's} Lemma
		\begin{equation}\label{ine-wts}
		\norm{\bs{w}(t)}^2_H \leq c_9 \norm{\bs{w}(s)}^2_H, \quad 0 < s < t < 2\ell,
		\end{equation}
		\nin where $c_9 = \text{exp}\left((t-s)\ell^{-1}\right) \leq \text{exp}(2c_7)$. In other words, the smallness of $\ell$ eliminates the (exponential) dependence of the Lipschitz constant of $S(t)$ on the viscosities.\\
		
		\nin Integrating \eqref{de-est-1} over $(s,2\ell)$, where $s \in (0,\ell)$ is fixed, one further derives
		\begin{equation*}
		c_8 \nu_1 \int_{s}^{2\ell} \norm{\bs{w} \tcr{(t)}}^2_{V_2} \tcr{dt} + c_8 \int_{s}^{2\ell} \int_{\Omega} I^2 (\bs{Du}, \bs{Dv}) \ dx\tcr{dt} \leq \norm{\bs{w}(s)}_H^2 + c_7 \ell^{-1} \int_{s}^{2\ell} \norm{\bs{w}(\tcr{t})}_H^2 \tcr{dt}.
		\end{equation*}
		\nin By \eqref{ine-wts}, we obtain $\ds \int_{s}^{2\ell} \norm{\bs{w}(t)}^2_H \ dt \leq 2 c_9 \ell \norm{\bs{w}(s)}_H^2$. By substituting this back in the above inequality, we obtain,
		\begin{equation*}
		c_8 \nu_1 \int_{s}^{2\ell} \norm{\bs{w}\tcr{(t)}}^2_{V_2}\tcr{dt} + c_8 \int_{s}^{2\ell} \int_{\Omega} I^2 (\bs{Du}, \bs{Dv}) \ dx\tcr{dt} \leq c_{10} \norm{\bs{w}(s)}_H^2.
		\end{equation*}
		\nin Integrating over $s \in (0,\ell)$ yields,
		\begin{equation*}
		\ell \nu_1 \int_{s}^{2\ell} \norm{\bs{w}\tcr{(t)}}^2_{V_2}\tcr{dt} + \ell \int_{s}^{2\ell} \int_{\Omega} I^2 (\bs{Du}, \bs{Dv}) \ dx\tcr{dt} \leq c_{11} \int_{0}^{\ell} \norm{\bs{w}(s)}_H^2\tcr{ds}.
		\end{equation*}
		\nin This proves \eqref{lip-L}. We also note here that
		\begin{equation}\label{est_IDD}
		\left[\int_{s}^{2\ell} \int_{\Omega} I^2 (\bs{Du}, \bs{Dv}) \ dx\tcr{dt}\right]^{\rfrac{1}{2}} \leq c_{12} \nu_1^{\frac{1}{2}}L_1 \left[ \int_{0}^{\ell} \norm{\bs{w}(s)}_H^2 \tcr{ds}\right]^{\rfrac{1}{2}}.
		\end{equation}
		To prove \eqref{lip-Lt}, \eqref{nse_L} is used to get
		\begin{align*}
		\norm{\partial_t \bs{w}}_{L^2(\ell, 2 \ell; V'_r)} &= \sup_{\varphi} \int_{\ell}^{2\ell} \ip{\partial_t \bs{w}}{\varphi}\tcr{dt},\\
		&= \sup_{\varphi} \bigg[ \underbrace{\int_{\ell}^{2\ell} \int_{\Omega} \left( \bs{S}(\bs{Du}) - \bs{S}(\bs{Dv}) \right):\bs{D}\varphi \,\tcr{dx}\tcr{dt}}_{I_1} + \underbrace{\int_{\ell}^{2\ell} \int_{\Omega} (\bs{u} \otimes \bs{u} - \bs{v} \otimes \bs{v}):\nabla \varphi \,\tcr{dx}\tcr{dt}}_{I_2} \\ &\hspace{4cm}  \tcr{ +\alpha \underbrace{\int_{\ell}^{2\ell} \int_{\partial \Omega} \left( \bs{s}(\bs{u}) - \bs{s}(\bs{v}) \right) \bs{w}:\varphi \,dSdt}_{I_3} }\bigg]
		\end{align*}
		\nin where the supremum is taken over $\varphi \in L^2(\ell, 2\ell; V_r)$ with $\norm{\varphi} = 1$. By H\"{o}lder inequality
		\begin{equation*}
		I_1 \leq \left[\int_{\ell}^{2\ell} \norm{\bs{S}(\bs{Du}) - \bs{S}(\bs{Dv})}^2_{r'} \tcr{dt} \right]^{\rfrac{1}{2}}.
		\end{equation*}
		Then by \eqref{prop_S}, \eqref{def-IDD} we obtain
		\begin{equation*}
		\abs{\bs{S}(\bs{Du}) - \bs{S}(\bs{Dv})} \leq I(\bs{Du}, \bs{Dv})\left( \nu_1 + \nu_2 \left( \abs{\bs{Du}} + \abs{\bs{Du}} \right)^{r-2} \right)^{\rfrac{1}{2}},
		\end{equation*}
		\nin hence
		\begin{align*}
		\norm{\bs{S}(\bs{Du}) - \bs{S}(\bs{Dv})}_{r'} &= \left[ \int_{\Omega} \abs{\bs{S}(\bs{Du}) - \bs{S}(\bs{Dv})}^{r'} \tcr{dx} \right]^{\rfrac{1}{r'}},\\
		&\leq \left[ \int_{\Omega} I^{r'}(\bs{Du}, \bs{Dv})\left( \nu_1 + \nu_2 \left( \abs{\bs{Du}} + \abs{\bs{Du}} \right)^{r-2} \right)^{\rfrac{r'}{2}} \tcr{dx}  \right]^{\rfrac{1}{r'}},\\
		&\leq \left[ \int_{\Omega} I^{sr'}(\bs{Du}, \bs{Dv}) \tcr{dx} \right]^{\rfrac{1}{sr'}} \left[ \int_{\Omega} \left( \nu_1 + \nu_2 \left( \abs{\bs{Du}} + \abs{\bs{Du}} \right)^{r-2} \right)^{\rfrac{s'r'}{2}}  \tcr{dx} \right]^{\rfrac{1}{s'r'}}.
		\end{align*}
		\nin We choose $s,s'$ such that $\ds \rfrac{1}{s} + \rfrac{1}{s'} = 1, r's = 2$ and $\ds r's' = \rfrac{2r}{r-2}$. Then we obtain
		\begin{equation*}
		\norm{\bs{S}(\bs{Du}) - \bs{S}(\bs{Dv})}_{r'} \leq \left[ \int_{\Omega} I^2(\bs{Du}, \bs{Dv}) \tcr{dx} \right]^{\rfrac{1}{2}} \underbrace{\left[ \int_{\Omega} \left( \nu_1 + \nu_2 \left( \abs{\bs{Du}} + \abs{\bs{Du}} \right)^{r-2} \right)^{\rfrac{r}{r-2}} \tcr{dx}  \right]^{\rfrac{r-2}{2r}}}_{M}.
		\end{equation*}
		\nin Now we compute,
		\begin{align*}
		M &= \nu_1^{\rfrac{1}{2}}\left[ \int_{\Omega} \left( 1 + \nu_1^{-1}\nu_2 \left( \abs{\bs{Du}} + \abs{\bs{Du}} \right)^{r-2} \right)^{\rfrac{r}{r-2}} \tcr{dx} \right]^{\rfrac{r-2}{2r}},\\
		&\leq c_{13} \nu_1^{\rfrac{1}{2}}\left[ 1 + \nu_1^{-1}\nu_2 \int_{\Omega} \left( \abs{\bs{Du}}^r + \abs{\bs{Du}}^r \right) \tcr{dx} \right]^{\rfrac{r-2}{2r}},\\
		&\leq c_{14} \nu_1 \left( 1 + M_r \right),
		\end{align*}
		cf \eqref{M_r}. Note that the integral above cannot be bounded directly by $B_r$ in \eqref{B_r}. But $U$ in Lemma \ref{lem-Br} is bounded by $B_r$. Combining \eqref{L_1}, \eqref{est_IDD}
		\begin{equation*}
		\tcr{I_1} \leq U \norm{\chi - \psi}_{H_{\ell}}.
		\end{equation*}
		\nin Now we proceed to the estimate
		\begin{equation*}
		\tcr{I_2} \leq \int_{\ell}^{2\ell} \int_{\Omega} \abs{\bs{w}} \left( \abs{\bs{u}} + \abs{\bs{v}} \right) \abs{\nabla \varphi} \tcr{dx}\tcr{dt} \leq \left[\int_{\ell}^{2\ell} \norm{\abs{\bs{w}} \left( \abs{\bs{u}} + \abs{\bs{v}} \right)} \tcr{dt} \right]^{\rfrac{1}{r'}}.
		\end{equation*}
		\nin Then we compute
		\begin{align*}
		\norm{\abs{\bs{w}} \left( \abs{\bs{u}} + \abs{\bs{v}} \right)}_{r'} &= \left[\int_{\Omega} \abs{\bs{w}}^{r'} \left( \abs{\bs{u}} + \abs{\bs{v}} \right)^{r'}\tcr{dx} \right]^{\rfrac{1}{r'}},\\
		&\leq c_{15} \norm{\bs{w}}_2 \left( \norm{\bs{u}}_{\frac{2r}{r-2}} + \norm{\bs{v}}_{\frac{2r}{r-2}} \right).
		\end{align*}
		\nin We consider two cases
		\begin{enumerate}[(i)]
			\item case $r \in (12/5,3]$. Using \eqref{bd-2r-3}, \eqref{B_r}, we obtain
			\begin{align*}
			I_2 &\leq c_{16} \left[ \int_{\ell}^{2\ell} \norm{\bs{w}}_2^2 \tcr{dt} \right]^{\rfrac{1}{2}} \sup_{t \in (\ell, 2 \ell)}  \left[ \norm{\bs{u}}_{\frac{2r}{r-2}} + \norm{\bs{v}}_{\frac{2r}{r-2}} \right],\\
			&\leq c_{17}B_0^{\frac{5r-12}{5r-6}}B_r^{\frac{6}{5r-6}}\left[ \int_{\ell}^{2\ell} \norm{\bs{w}}_2^2 \tcr{dt} \right]^{\rfrac{1}{2}},\\
			&\leq W\norm{\chi - \psi}_{H_{\ell}}.
			\end{align*}
			\nin satisfying the first part of \eqref{W}.
			\item case $r > 3$. Using \eqref{bd-2r-2}, \eqref{B_r}, we obtain
			\begin{align*}
			I_2 &\leq c_{16} \left[ \int_{\ell}^{2\ell} \norm{\bs{w}}_2^2 \tcr{dt} \right]^{\rfrac{1}{2}} \sup_{t \in (\ell, 2 \ell)}  \left[ \norm{\bs{u}}_{\frac{2r}{r-2}} + \norm{\bs{v}}_{\frac{2r}{r-2}} \right],\\
			&\leq c_{18}B_0^{\frac{r-2}{r}}B_r^{\frac{2}{r}}\left[ \int_{\ell}^{2\ell} \norm{\bs{w}}_2^2 \tcr{dt} \right]^{\rfrac{1}{2}},\\
			&\leq W\norm{\chi - \psi}_{H_{\ell}}.
			\end{align*}
			\nin \nin satisfying the second part of \eqref{W}. \tcr{ Finally, we estimate 
				\begin{equation*}
				I_3 \leq Q \norm{\chi - \psi}_{H_{\ell}},
				\end{equation*}
				\nin with \eqref{s-prop}. This concludes the proof of the Lemma.}
			
		\end{enumerate}
	\end{proof}
	\nin Now we formulate the main result.
	\begin{thm} \label{thm:main}
		Let the stress tensor satisfy \eqref{prop_S}, \eqref{prop_Phi} with $\ds r > \rfrac{12}{5}$. Then \eqref{nse_L} has a global attractor $\calA$, and its dimension can be estimated as 
		\begin{equation}
		d^H_f \left( \calA \right) \leq \tcr{c_{19} \left( L_1^4 + \ell L_1^{\frac{2(11r-6)}{3r}}L_2 \right) \ln L_1},
		\end{equation}
		where $L_1,\ L_2$ and $\ell$ are given in Lemma \ref{lem-L}.
	\end{thm}
	\begin{proof}
		Follows exactly along the arguments of \cite[Theorem
		4.1]{BEKP:2009}, using the estimates of Lemma~\ref{lem-L} above
		and Lemma~\ref{lem-cover} below.
	\end{proof}
	
	\appendix
	\section{Coverings and Fractal Dimension.}\label{appen-cov}
	Now we present an elementary description of a class of Sobolev
	and Bochner spaces with fractional derivatives. These
	formulations will be used to obtain covering numbers for
	compact embeddings. We follow a similar technique used in
	\cite[Section 7: Appendix]{BEKP:2009}, or \cite[Secion 4]{BMP:2005}. Consider the following inhomogeneous Stokes problem,
	\begin{subequations}
		\begin{empheq}[left=\empheqlbrace]{align}
		\partial_t \bs{v} - \div \ \bs{D} \bs{v} + \nabla \pi = \bs{f}, \quad \div \ \bs{v} &= 0 &\text{ in } Q,\\
		\bs{v} \cdot \bs{n} = 0, \quad -(\bs{D}\bs{v})\bs{n} + \alpha \bs{v} &= \beta \partial_t \bs{v} &\text{ on } \Gamma,\\
		\bs{v}(0) &= \bs{v}_0 &\text{ in } \overline{\Omega}.
		\end{empheq}
	\end{subequations}
	\tcr{\nin Then the above dynamical system defines the operator $\calA$ which generates a strongly continuous analytic semigroup on $H$ with a compact resolvent with domain $\calD(\calA) \subset \subset V$, see \cite[Theorem 1, p. 7]{PZ:2022a}. We thus have linear (unbounded) operator $\calA: \calD(\calA) \to H$ satisfying,
		\begin{equation}
		\left(\bs{u}, \bs{\varphi}\right)_V = \left( \calA \bs{u}, \bs{\varphi} \right)_H, \ \forall \bs{u} \in \calD(\calA), \ \forall \varphi \in V.
		\end{equation}
		Moreover, from the same reference we have that $\calA$ is surjective, and is also symmetric on its domain, i.e. for any $\bs{u}, \bs{v} \in \calD(\calA)$ we have
		\begin{equation}
		\left( \calA \bs{u}, \bs{v} \right)_H = \left( \bs{u}, \calA\bs{v} \right)_H.
		\end{equation}}
	\nin Then by virtue of \cite[Section 5, p. 168]{BPDM:2007}, we can define the domains fractional powers of the operator $\calA$. Then by \cite[Theorem 1.15.3, p. 114]{HT:1980} for $0 \leq \theta \leq 1$,
	\begin{equation}
	\calD(\calA^{\theta}) \subset \left[L^2(\Omega) \times L^2(\partial \Omega), H^2(\Omega) \times L^2(\partial \Omega) \right]_{\theta} \hookrightarrow H^{2\theta}(\Omega) \times L^2(\partial \Omega),
	\end{equation}
	
	\nin where $H^{s}(\Omega) = W^{s,2}(\Omega), \ s \in \BR$. For more details on domains of fractional powers of matrix-valued operators, we refer to \cite{LT:2015} and references therein. Let $\bs{w}_j, \lambda_j = 1,2,\dots$ be the eigenfunctions and eigenvalues of the operator $\calA$ respectively.
	\begin{subequations}\label{eigen}
		\begin{align}
		-\div \ \bs{D}\bs{w}_j &= \lambda_j \bs{w}_j,  &\text{ in } \Omega\\ 
		\div \ \bs{w}_j &= 0, &\text{ in } \Omega\\
		\bs{D}\bs{w}_j\bs{n} + \alpha \bs{w}_j &= \lambda_j \beta \bs{w}_j &\text{ in } \Gamma
		\end{align}
	\end{subequations}
	\nin Note that we have taken $\partial_t \bs{v} = -\lambda_i \bs{v}$ for $\lambda_i$ to be nonnegative. One can show that $\{ \bs{w}_j \}_{j \in \BN}$ is a basis for $V$ and $H$, it is orthogonal in $V$ and orthonormal in $H$. Moreover, we have $\lim_{i \to \infty} \lambda_i = + \infty$. See \cite[Lemma 3.1]{Mar:2019}. We also have 
	\begin{equation}\label{ub-lam}
	\tcr{C j^{\rfrac{1}{2}} \leq \lambda_j \leq \wti{C} j^{\rfrac{2}{3}},}
	\end{equation}
	for dimension $d = 3$ by \cite[Section 3.2]{PZ:2022a}\tcr{, and Lemma \ref{lam-ub} for some positive constants $C$, $\wti{C}$}. For $b \in \BR$, one introduces the space $\mathbb{H}^b := H^b(\Omega) \times L^2(\partial \Omega)$ as
	\begin{equation}
	\ds \mathbb{H}^b = H^b(\Omega) \times L^2(\partial \Omega).
	\end{equation}
	Let us define, $\mathbb{H}^{-b} := \left( H^b(\Omega) \right)' \times L^2(\partial \Omega)$, with the duality given by the generalized scalar product in $H$. Further, we define $\mathbb{H}^b$ as a class of interpolation spaces in the sense that $\ds \left[ \mathbb{H}^{b_1}, \mathbb{H}^{b_2} \right]_{\alpha} = \mathbb{H}^b$, where $b = (1 - \alpha)b_1 + \alpha b_2$. To relate $\mathbb{H}^b$ to classical Sobolev spaces (product), observe that
	\begin{align*}
	\norm{(\bs{u}, \bs{g})}_{\mathbb{H}^0}^2 &= \norm{(\bs{u}, \bs{g})}^2_H = \norm{\bs{u}}^2_2 + \beta \norm{\bs{g}}^2_{L^2(\partial \Omega)},\\
	\ipp{\calA \bs{u}}{\bs{u}}_{H} &= \norm{\calA^{\rfrac{1}{2}}\bs{u}}^2_{H} = \sum_{j} a^2_j \lambda_j, \\
	&= \norm{\bs{v}}^2_{H^1(\Omega)} + \beta \norm{tr \ \bs{u}}^2_{L^2(\partial \Omega)} \sim \norm{\bs{u}}_{\mathbb{H}^1}^2 \sim \norm{\bs{u}}^2_V, \text{ and}\\
	\text{ compute } \calA \bs{u} &= \sum_{j} a_j\calA \bs{w}_j = \sum_j a_j\lambda_j \bs{w}_j, \text{ hence } \norm{\calA \bs{u}}_{H}^2 = \norm{\bs{u}}_{\mathbb{H}^2}^2.
	\end{align*}  
	
	\nin Similarly, an orthonormal basis for $L^2(0,\ell)$ will be defined as
	\begin{equation}
	\varphi_0(t) = {\ell}^{-\rfrac{1}{2}}, \quad \varphi_k(t) = 2^{\rfrac{1}{2}}\ell^{-\rfrac{1}{2}}\cos(k\pi t \ell^{-1}), \quad k \geq 1.
	\end{equation} 
	\nin One sets $\mu_0 = \ell^{-2}, \ \mu_k = k^2 \pi^2 \ell^{-2}$. The space $H^a(0,\ell)$ is defined as
	\begin{equation*}
	\norm{\phi}^2_{H^a(0,\ell)} = \sum_k a_k^2\mu_k^a, \quad a_k = \int_{0}^{\ell}\phi(t)\varphi_k(t)dt.
	\end{equation*}
	The seminorm $\dot{H}^a(0,\ell)$ will also be used, 
	\begin{equation*}
	\norm{\phi}^2_{\dot{H}^a(0,\ell)} = \sum_{k \neq 0} a_k^2\mu_k^a,
	\end{equation*}
	and the space $H^a_0(0,\ell)$, in the definition of which $\varphi_k(t)$s are replaced by 
	\begin{equation*}
	\psi_k(t) = 2^{\rfrac{1}{2}}\ell^{-\rfrac{1}{2}}\sin(k\pi t \ell^{-1}), \quad k \geq 0.
	\end{equation*} 
	\nin Note that 
	\begin{equation}\label{mu}
	\begin{aligned}
	\mu_k \sim k^2 \ell^{-2},\\
	\abs{\varphi_k(t)}, \abs{\psi_k(t)} \leq c_1 \ell^{-\rfrac{1}{2}}.
	\end{aligned}
	\end{equation}
	\nin The dependence on $\ell$ has to be carefully traced down, since $\ell << 1$ in the applications.\\ 
	
	\nin Now we combine $\bs{w}_j, \varphi_k$ to describe certain norms of fractional Bochner spaces. For $\bs{u}(x,t):\Omega \times (0,\ell) \longrightarrow \BR^3$, one sets
	\begin{align*}
	\norm{\bs{u}}^2_{H^a(0,\ell;\mathbb{H}^b)} &= \sum_{j,k} a_j^2 \lambda_j^b \mu_k^a,\\
	\norm{\bs{u}}^2_{\dot{H}^a(0,\ell;\mathbb{H}^b)} &= \sum_{j,k \neq 0} a_j^2 \lambda_j^b \mu_k^a,\\
	\text{ where } a_{jk} &= \int_{\Omega \times \partial \Omega \times (0,\ell)} \bs{v}(x,t)\cdot\bs{w}_j(x)\varphi_k(t)dxdt.
	\end{align*}
	\nin As above, there is the introduction $H^a_0(0,\ell;\mathbb{H}^b)$ using $\psi_k$ in place of $\varphi_k$. It is straightforward to verify that 
	\begin{equation}\label{eq-nr-H}
	\norm{\bs{u}}_{L^2(0,\ell;\mathbb{H}^b)} = \norm{\bs{u}}_{H^0(0,\ell;\mathbb{H}^b)} = \norm{\bs{u}}_{H^0_0(0,\ell;\mathbb{H}^b)}.
	\end{equation} 
	In the following Lemma from \cite[Lemma 7.1]{BEKP:2009}, it is proven that the seminorm $\dot{H}^1(0,\ell)$ can be estimated in terms of the time derivative. The value of $b$ given in \eqref{eq-b} is obtained by \eqref{emb}. 
	
	\begin{lemma}\label{lem-emb-H}
		Let $r \geq 2$ and let $b$ be given by 
		\begin{equation}\label{eq-b}
		b = \frac{5r - 6}{2r}, \ \text{i.e. } b \geq 1.
		\end{equation}
		Then 
		\begin{equation*}
		\norm{\bs{u}}_{\dot{H}^1(0,\ell;\mathbb{H}^{-b})} \leq c_1 \norm{\partial_t \bs{u}}_{L^2(0,\ell;V_r')}
		\end{equation*}
		Here $\partial_t$ stands for the distributional derivative in $\Omega \times (0,\ell)$.
	\end{lemma}
	\nin Then we obtain the following two Lemmas from \cite[Lemma 7.2, Lemma 7.3]{BEKP:2009} by devising similar computations. 

	\begin{lemma}\label{Lem-7.2}
		Let $r \geq 2,\ \ell > 0$ and $C_1, C_2 >> 1$. Denote
		\begin{equation*}
		\calM = \left\{ \bs{u}: \norm{\bs{u}}_{L^2(0,\ell;V_2)} \leq C_1, \ \norm{\partial_t\bs{u}}_{L^2(0,\ell;V'_r)} \leq C_2 \right\}.
		\end{equation*}
		There exists orthonormal projection $\calP$ in $L^2(0,\ell;H)$ such that 
		\begin{equation}\label{est-dist}
		\text{dist}\left( \calM, \calP(\calM) \right) \leq \frac{1}{\sqrt{8}},
		\end{equation}
		and 
		\begin{equation}\label{rankP}
		\text{rank } \calP \leq c_2 \tcr{\left( C_1^4 + \ell C_1^{\frac{2(11r - 6)}{3r}}C_2 \right)}
		\end{equation}
	\end{lemma}
	\begin{proof}
		The proof of this lemma follows similar argumentation as \cite[Lemma 7.2]{BEKP:2009}. By virtue of \eqref{eq-nr-H} and Lemma \ref{lem-emb-H}, $\calM$ can described by $\ds H^0(0,\ell;\mathbb{H}^1)$  and $\ds H^0(0,\ell;\mathbb{H}^{-b})$, where $\ds b = \rfrac{(5r-6)}{2r}$. Then the Fourier coefficients of $\bs{u} \in \calM$ satisfy 
		\begin{equation}
		\sum_{j,k} a_{jk}^2 \lambda_j \leq c_3 C_1^2, \quad \sum_{j,k \neq 0} a_{jk}^2 \lambda_j^{-b}\mu_k \leq c_4 C_2^2.
		\end{equation}
		\nin Hence it is enough to take $\calP$ as the projection to the span of 
		\begin{equation*}
		\left\{ \bs{w}_j\varphi_k: \ \lambda_j \leq 8c_3C_1^2 \text{ and } \mu_k \leq 8c_4 \lambda_j^bC_2^2 \right\}.
		\end{equation*}
		\nin First, we show that \eqref{est-dist} holds. First observe that, for $\bs{u} \in \calM$,
		\begin{equation*}
		\bs{u} = \sum_{j,k} a_{jk}\bs{w}_j\varphi_k, \ \text{ and } \quad \calP \bs{u} = \sum_{\left\{ \lambda_j \leq 8c_3C_1^2,\ \mu_k \leq 8c_4 \lambda_j^bC_2^2 \right\}} a_{jk}\bs{w}_j\varphi_k.
		\end{equation*}
		Now we estimate
		\begin{equation*}
		\norm{\bs{u} - \calP \bs{u}}^2_{H} = \sum_{\left\{ \lambda_j > 8c_3C_1^2,\ \text{ or } \mu_k > 8c_4 \lambda_j^bC_2^2 \right\}} a_{jk}^2.
		\end{equation*}
		\nin We further estimate above two different cases separately,
		\begin{equation*}
		\norm{\bs{u} - \calP \bs{u}}^2_{H} = \sum_{\left\{ \lambda_j > 8c_3C_1^2 \right\}} a_{jk}^2, \ \text{ or } \ \norm{\bs{u} - \calP \bs{u}}^2_{H} = \sum_{\left\{ \mu_k > 8c_4 \lambda_j^bC_2^2 \right\}} a_{jk}^2.
		\end{equation*}
		\nin Furthermore,
		\begin{equation*}	
		\norm{\bs{u} - \calP \bs{u}}^2_{H} =\sum_{\left\{ \lambda_j > 8c_3C_1^2 \right\}} a_{jk}^2 \lambda_j \frac{1}{\lambda_j}, \ \text{ or }  \norm{\bs{u} - \calP \bs{u}}^2_{H} = \sum_{\left\{ \mu_k > 8c_4 \lambda_j^bC_2^2 \right\}} a_{jk}^2 \lambda_j^{-b}\mu_k \frac{\lambda_j^b}{\mu_k},
		\end{equation*}
		\nin combining both cases, we obtain
		\begin{equation}
		\norm{\bs{u} - \calP \bs{u}}^2_{H} \leq \frac{1}{8}.
		\end{equation}
		\nin Hence \eqref{est-dist}. Now we recall \eqref{ub-lam} and \eqref{mu}, we estimate
		\tcr{ 
			\begin{align}
			\text{rank } \calP &\leq \sum_{\left\{j \ \leq \ c_5C_1^4 \right\}} \left( 1 + C_2 \ell j^{\rfrac{b}{3}} \right) \leq c_{6} \left( C_1^4 + \ell C_2 C_1^{\frac{4(b + 3)}{3}} \right), \nonumber\\
			&= c_6 \left( C_1^4 + \ell C_1^{\frac{2(11r - 6)}{3r}}C_2 \right),
			\end{align}
		}
	\end{proof}
	\begin{lemma}	\label{lem-cover}
		The set $\calM$ from Lemma \ref{Lem-7.2} can be covered by $K$ balls of radii $\ds \rfrac{1}{2}$ in $L^2(0,\ell;H)$, where
		\begin{equation}\label{lnK}
		\ln K \leq \tcr{c_7 \left( C_1^4 + \ell  C_1^{\frac{2(11r - 6)}{3r}}C_2 \right)\ln C_1}.
		\end{equation}
	\end{lemma}
	
	\begin{rmk}
		The difference between the estimate obtained in \cite[Lemma 7.2]{BEKP:2009} and \eqref{rankP} is due to the difference between the lower \tcr{and upper} bounds of the eigenvalues in two cases. In the former, the authors had $\lambda_j \sim c_1j^{\rfrac{2}{3}}$, and in our case we have $\tcr{C j^{\rfrac{1}{2}} \leq \lambda_j \leq \wti{C} j^{\rfrac{2}{3}}}$. This difference is also evident in the estimate \eqref{lnK}.
	\end{rmk}
	
	\section{Appendix}
	\begin{lemma}\cite[Lemma II.2.33, p. 66]{BF:2013}\label{lem-int}
		Let $\Omega$ be any open set of $\BR^d$ and let $u \in L^p(\Omega) \cap L^q(\Omega)$ with $1 \leq p,q \leq \infty$. Then for all $r$ such that 
		\begin{equation*}
		\frac{1}{r} = \frac{\theta}{p} + \frac{1-\theta}{q}, \ 0 < \theta < 1.
		\end{equation*}
		we have $u \in L^r(\Omega)$, and 
		\begin{equation}\label{int-Lp}
		\norm{u}_r \leq \norm{u}^{\theta}_p \norm{u}^{1 - \theta}_q.
		\end{equation}
	\end{lemma}
	\begin{thm}\cite[p. 173]{BF:2013}
		Let $\Omega$ be a Lipschitz domain in $\BR^d$ with compact boundary. Let $p \in [1,\infty]$ and $\ds q \in \left[p, \frac{pd}{d - p} \right]$. There exists a $C > 0$ such that
		\begin{equation}\label{int-inq}
		\norm{\varphi}_{L^q(\Omega)} \leq C \norm{\varphi}_{L^p(\Omega)}^{1 + \rfrac{d}{q} - \rfrac{d}{p}} \norm{\varphi}_{W^{1,p}(\Omega)}^{\rfrac{d}{p} - \rfrac{d}{q}}, \quad \text{ for all } \varphi \in W^{1,p}(\Omega).
		\end{equation}
	\end{thm}
	
	\begin{lemma}\cite[Lemma 1.11, p. 63]{BMR:2007}\label{lem-Korn}
		Let $\Omega \in \calC^{0,1}$ and $q \in (0,+\infty)$. Then there exists a positive constant $C$, depending only on $\Omega$ and $q$, such that for all $\bs{v} \in W^{1,q}(\Omega)$ which has the trace $tr \ \bs{v} \in L^2(\partial \Omega)$, the following inequality hold,
		\begin{align*}
		\norm{\bs{w}}_{W^{1,q}(\Omega)} &\leq C \left( \norm{\bs{Dv}}_{L^q(\Omega)} + \norm{tr \ \bs{v}}_{L^2(\partial \Omega)} \right),\\
		\norm{\bs{w}}_{W^{1,q}(\Omega)} &\leq C \left( \norm{\bs{Dv}}_{L^q(\Omega)} + \norm{\bs{v}}_{L^2(\Omega)} \right).
		\end{align*}
	\end{lemma}
	
	\begin{thm}\cite[p328]{HT:1980}
		Let $\Omega$ be an arbitrary bounded domain, $\Omega \subset \BR^d$. Let $0 \leq t \leq s < \infty$ and $ \infty > q \geq \wti{q} > 1$. Then, the following embedding holds true:
		\begin{equation}\label{emb}
		W^{s,\wti{q}}(\Omega) \subset W^{t, q}(\Omega), \ s - \frac{d}{\wti{q}} \geq t - \frac{d}{q} \quad \qedsymbol
		\end{equation}
	\end{thm}
	
	\begin{lemma}\label{lem-int-w}
		Let $\Omega \subset \BR^3$ be a bounded domain. Then
		\begin{equation}
		\norm{u}_6 \leq c_0 \norm{u}_2^{\frac{1}{3}}\norm{u}_{W^{1,3}(\Omega)}^{\frac{2}{3}},
		\end{equation}
		for any function $u \in W^{1,3}(\Omega)$.
	\end{lemma}
	\begin{proof}
		By interpolation result \eqref{int-inq}, we obtain
		\begin{equation}
		\norm{u}_6 \leq c_1 \norm{u}_3^{\frac{1}{2}}\norm{u}_{W^{1,3}(\Omega)}^{\frac{1}{2}}.
		\end{equation}
		Then by \eqref{int-Lp}, we obtain
		\begin{equation}
		\norm{u}_3 \leq \norm{u}_2^{\frac{1}{2}}\norm{u}_6^{\frac{1}{2}}.
		\end{equation}
		\nin By combining above two inequalities, we obtain the result.
	\end{proof}

	\tcr{
		\begin{lemma}\label{lam-ub}
			Let the dimension of $\Omega$ be $d$. Then the eigenvalues $\{ \lambda_j \}$ of the problem \eqref{eigen} are bounded above by $cj^{\rfrac{2}{d}}$ where $c> 0$.
		\end{lemma}
		\begin{proof}
			The asymptotic behavior of the eigenvalues $\lambda_j$ as $j \to \infty$ can be estimated using the Rayleigh quotient
			\begin{equation}
			\calR(\bs{u}) = \frac{\ds \int_{\Omega} \abs{\bs{Du}}^2dx + \alpha\int_{\partial \Omega} \abs{\bs{u}}^2dS}{\ds \int_{\Omega} \abs{\bs{u}}^2dx + \beta \int_{\partial \Omega} \abs{\bs{u}}^2dS}.
			\end{equation}
			\nin With this notion we have
			\begin{equation}
			\lambda_j = \inf_{M \in X_j(V)} \sup_{\bs{u} \in M \backslash \{0\}} \calR(\bs{u}),
			\end{equation}
			\nin where $X_j(V)$ is the $j$-dimensional subspaces of the space $V$ with divergence free condition and zero normal component. Then we estimate 
			\begin{equation}
			\calR(\bs{u}) \leq c \frac{\norm{\bs{u}}^2_{W^{1,2}(\Omega)}}{\norm{\bs{u}}^2_2}. 
			\end{equation}
			\nin Therefore
			\begin{equation}
			\lambda_j \leq c \inf_{M \in X_j(W)} \sup_{\bs{u} \in M \backslash \{0\}} \frac{\norm{\bs{u}}^2_{W^{1,2}(\Omega)}}{\norm{\bs{u}}^2_2} = c \mu_k,
			\end{equation}
			\nin where space $W$ with divergence free and zero boundary conditions, i.e. $W \subset V$. Now this upper-bound $c\mu_k$ is related to the following Stokes-eigenvalue problem
			\begin{align*}
			\Delta \bs{u} + \nabla \pi &= \mu_k \bs{u} \quad \text{ in } \Omega,\\
			\div \ \bs{u} &= 0 \quad \text{ in } \Omega,\\
			\bs{u} &= 0 \quad \text{ on } \partial \Omega.
			\end{align*}
			\nin It is shown in \cite{CF:1988} that $\mu_k \sim k^{\rfrac{2}{d}}$. Hence we have $\lambda_j \leq c j^{\rfrac{2}{d}}$.
	\end{proof}}
	
	\tcr{The authors have no conflicts of interest to declare that are relevant to the content of this article. Authors would like to thank anonymous referees for their valuable comments.}
	
	\bibliographystyle{plain}
	\bibliography{Dynbc_arXiv}
	
\end{document}